\documentclass[11pt]{amsart}

\usepackage{amssymb, amsthm, amsmath}
\usepackage{xcolor}
\newtheorem{theorem}{Theorem}[section]
\newtheorem{claim}[theorem]{Claim}
\newtheorem{lemma}[theorem]{Lemma}
 
\newtheorem{prop}[theorem]{Proposition}

\newtheorem*{theorem*}{Theorem}{\bf}{\it}
\newtheorem*{proposition*}{Proposition}{\bf}{\it}
\newtheorem*{observation*}{Observation}{\bf}{\it}
\newtheorem*{lemma*}{Lemma}{\bf}{\it}

\theoremstyle{definition}

\theoremstyle{remark}
\newtheorem{remark}[theorem]{Remark}

\renewcommand{\tilde}{\widetilde}

\def\Xint#1{\mathchoice
{\XXint\displaystyle\textstyle{#1}}%
{\XXint\textstyle\scriptstyle{#1}}%
{\XXint\scriptstyle\scriptscriptstyle{#1}}%
{\XXint\scriptscriptstyle\scriptscriptstyle{#1}}%
\!\int}
\def\XXint#1#2#3{{\setbox0=\hbox{$#1{#2#3}{\int}$ }
\vcenter{\hbox{$#2#3$ }}\kern-.6\wd0}}

\def\dashint{\Xint-}
\newcommand{\N}{\mathcal{N}}
\begin{document}
\title[]{ Nodal sets of Laplace eigenfunctions:
proof of Nadirashvili's conjecture and of the lower bound in Yau's conjecture.}
\author{Alexander Logunov}
\address{School of Mathematical Sciences, Tel Aviv University, Tel Aviv 69978, Israel}
\address{Chebyshev Laboratory, St. Petersburg State University, 14th Line V.O., 29B, Saint Petersburg 199178 Russia}
\email{log239@yandex.ru}


\begin{abstract}
  Let $u$ be a harmonic function in the unit ball $B(0,1) \subset \mathbb{R}^n$, $n \geq  3$, such that $u(0)=0$.
  Nadirashvili conjectured that there exists a positive constant $c$, depending  on the dimension $n$ only, such that $H^{n-1}(\{u=0 \}\cap B) \geq c$.  We prove Nadirashvili's conjecture as well as its counterpart on $C^\infty$-smooth Riemannian manifolds.
 The latter yields the lower bound in Yau's conjecture. Namely, we show that for any compact $C^\infty$-smooth Riemannian manifold $M$ (without boundary) of dimension $n$  there exists $c>0$ such that for any Laplace eigenfunction $\varphi_\lambda$ on $M$, which corresponds to the eigenvalue $\lambda$, the following inequality holds: $c \sqrt \lambda \leq H^{n-1}(\{\varphi_\lambda =0\})$.

 \end{abstract}
\maketitle
\section{Introduction.}
 Let $M$ be a $C^\infty$ smooth Riemannian  manifold (with or without boundary) of dimension $n$. 
 Let $B$ be a geodesic ball on $M$ with radius $1$. Assume that a number $\lambda >0$. Consider any solution of the equation  $\Delta u + \lambda u =0$ in $B$  (the boundary conditions for $u$ do not matter) and denote the zero set of $u$ by $Z_u$. We prove the following result:
\begin{theorem} \label{Yau}  There exist $c>0$ and $\lambda_0$, depending on $M$ and $B$ only, such that 
 if $\lambda> \lambda_0$, then 
$$ c \sqrt \lambda \leq H^{n-1}( Z_u\cap B).$$
\end{theorem}
  We prove a similar result for harmonic functions, which was conjectured by Nadirashvili (\cite{N}):
 \begin{theorem} \label{Nad}
 There exists $c>0$, depending on $M$ and $B$ only, such that for any harmonic function $h$ on $B$ that vanishes at the center of $B$ the following estimate holds: 
$$c  \leq H^{n-1}( Z_h\cap B).$$
 \end{theorem}
 As an immediate corollary from Theorem \ref{Nad} we obtain that if $h$ is a non-constant harmonic function in $\mathbb{R}^3$, then the zero set of $h$ has an infinite area.  Apparently, this is also a new result.

 Theorems \ref{Yau} and \ref{Nad} are related to each other by  a standard trick that allows to pass from Laplace eigenfunctions to harmonic functions.
If $u$  satisfies $\Delta u + \lambda u =0$ on $M$, then one can consider the harmonic function 
$$ h(x,t) = u(x) \exp(\sqrt \lambda t), $$
on the product manifold  $M \times \mathbb{R}$. The zero set  of $h$ and the zero set of $u$ are related by   $$Z_h=Z_u  \times \mathbb{R}.$$

 Then Theorem \ref{Yau} will follow in a straightforward way from the $\frac{1}{\sqrt \lambda}$-scaled version of Theorem \ref{Nad} and the fact 
 that $Z_u$ is $\frac{const}{\sqrt \lambda}$ is dense in $B$. The latter fact, which is well known, is the corollary of the Harnack inequality for harmonic functions. We bring the proof of this fact for the reader's convenience in Section \ref{sec:Yau}, where we also deduce Theorem \ref{Yau} from the scaled version of Theorem \ref{Nad}.

Most of the paper is devoted to the proof  of Nadirashvili's conjecture.

 Nadirashvili's conjecture was motivated by the question of Yau, who conjectured that if $M$ is a compact $C^\infty$-smooth Riemannian manifold  with no boundary, then there exist $c,C>0$, depending on $M$ only, such that the Laplace eigenfunctions $\varphi_\lambda$ on $M$ ($\varphi_\lambda$ corresponds to the eigenvalue $\lambda$) satisfy
 $$ c \sqrt \lambda \leq H^{n-1}(\varphi_\lambda=0) \leq C \sqrt \lambda.$$ 

 The lower bound for Yau's conjecture in dimension $2$, which is not difficult, was proved by Br\"uning and also by Yau.
 In dimension $n\geq 3$ the lower bound for Yau's conjecture  follows now from Theorem \ref{Yau}.

For the case of real-analytic metrics the Yau conjecture was proved by Donnelly and Fefferman \cite{DF}. 
 Theorem \ref{Yau} and Theorem \ref{Nad} do not follow from the Donnelly-Fefferman argument and are new in the case $M=\mathbb{R}^n$, $n\geq 3$, endowed with the standard Euclidean metric. Roughly speaking, Nadirashvili's conjecture implies the lower bound for Yau's conjecture and gives an additional information on small scales.  The assumption of real analyticity of the metric seems to be of no help for the question of Nadirashvili, but it was exploited by Donnelly and Fefferman to establish the lower and upper bounds  in Yau's conjecture.

 Concerning the upper bounds for  Yau's conjecture without real-analyticity assumptions, Donnelly and Fefferman (\cite{DF1}) proved that 
 in dimension $n=2$ the following estimate holds
 $$H^{1}(u=0) \leq C \lambda^{3/4}.$$
 Recently this upper bound was refined to $ C \lambda^{3/4 -\varepsilon}$ in \cite{LM1}, which we  advise to read before this paper.

 In higher dimensions  Hardt and Simon (\cite{HS}) showed that 
 $$H^{n-1}(u=0) \leq C \lambda^{C \sqrt \lambda}.$$
 Recently an upper bound with polynomial growth was obtained in \cite{LM2}:
 $$H^{n-1}(u=0) \leq C \lambda^{C }.$$

 In this paper we use  techniques of propagation of smallness developed in \cite{LM1} and \cite{LM2}. This paper is self contained with the exception of Theorem \ref{th1/2}, which was borrowed from \cite{LM2}.
 \subsection*{Acknowledgments} 
 This work  was started in collaboration with Eugenia Malinnikova who suggested to 
 apply the combinatorial approach to nodal sets of Laplace eigenfunctions. Her role in this work is no less  than the author's one.  Unfortunately, she
refused to be a coauthor of this paper.
On various stages of this work  
I discussed
it with Lev Buhovsky and Mikhail Sodin.  Lev and Mikhail also  
read the first
draft of this paper and made many suggestions and comments. I would like to mention that it was  
Dmitry Chelkak
from whom I heard for the first time about the question on the area of  
the zero set
of a harmonic function in $\mathbb{R}^3$. I thank all of them.


   This work was started while the author was visiting NTNU, continued at the Chebyshev Laboratory (SPBSU) and finished at TAU. The final version of the paper was completed at the Institute for Advanced Study. I am grateful to these institutions for their hospitatility and for great working conditions.
  
  The author was  supported in part by ERC Advanced Grant~692616 and ISF Grants~1380/13, 382/15 and by a Schmidt Fellowship at the Institute for Advanced Study.

\section{Almost monotonicity of the frequency.}
Given a point $O$ on a Riemannian manifold $M$,  let us consider normal coordinates with center at  $O$.  We will identify a neighborhood of $O$ on $M$ with a neighborhood of the origin in the Euclidean space. Now, we have two metrics:
the Euclidean metric, which we will denote by $d(x,y)$, and the Riemannian metric $d_g(x,y)$. The symbol $B(x,r)$ will denote the ball with center at $x$ and radius $r$ in Euclidean metric while $B_g(x,r)$ is used for the geodesic ball with respect to $g$. The radius $r$ will always be smaller than the injectivity radius. 
Due to the choice of normal coordinates for any $\varepsilon>0$ there is a sufficiently small $R_0=R_0(\varepsilon,M,g,O)>0$
such that 
 \begin{equation} \label{eq:epsilon} 
\frac{d(x,y)}{d_g(x,y)} \in(1-\varepsilon, 1+\varepsilon) 
\end{equation}
 for any two distinct points $x,y$ in $B_g(O,R_0)$. 
 We will always assume that $R_0$ is sufficiently small. In particular, we assume that \eqref{eq:epsilon} holds  with $\varepsilon= 1/2$.

 Throughout the paper the words "cube" and "box" (hyperrectangle)  will be used in the standard  Euclidean sense. The  reason why we need two metrics, but not one, is because we  will frequently partition cubes into smaller cubes and the combinatorial geometry ideas  are easier to describe  in $\mathbb{R}^n$ than on a manifold. We kindly advise the reader to think that $M$ is $\mathbb{R}^n$ to throw away half of the used notations and remove $\varepsilon$ error term in the monotonicity property for the frequency function defined below.

 Let $u$ be a harmonic function on $M$.
 Given a ball $B_g(x,r)$, define the function 
$$H(x,r)= \int\limits_{\partial B_g(x,r)} |u|^2 d S_r$$
where $S_r$ is the surface measure on $\partial B_g(x,r)$.

 We will use a slightly non-standard definition of the frequency function:
$$\beta(x,r) = \frac{rH'(x,r)}{2H(x,r)}.$$
Our definition is slightly different to the one in \cite{HL}, \cite{GL}, \cite{L}, in particular in the case of ordinary harmonic functions in $\mathbb{R}^n$  we don't normalize $H(r)$ by the total surface area $|S_r|$. Sometimes we will specify the dependence of $\beta$ and $H$ on $u$ and write $H_u(x,r)$ and $\beta_u(x,r)$.
  The frequency is almost monotonic  in the following sense:

For any $\varepsilon>0$ there exists $R_0 >0$ such that
if $r_1<r_2<R_0$ and $d_g(x,O)< R_0$, then
 $$ \beta(x,r_1) \leq \beta(x,r_2)(1+\varepsilon).$$
 See also Remark (3) to Theorem 2.2 in \cite{M}.
 
 It follows directly from the definition that 
\begin{equation} \label{eq:intbeta}
 H(x,r_2)/ H(x,r_1) =\exp( 2 \int_{r_1}^{r_2} \beta(x,r) d \log r )
\end{equation}
 and by the almost monotonicity property 
\begin{equation}
\left(\frac{r_2}{r_1}\right)^{2\beta(x,r_1)/(1+\varepsilon)} \leq \frac{H(x,r_2)}{H(x,r_1)} \leq   \left(\frac{r_2}{r_1}\right)^{2\beta(x,r_2)(1+\varepsilon)}.
\end{equation}
 
\section{A lemma on monotonic functions} 
\begin{lemma} \label{lem:mf}
 Let $f$ be a non-negative, monotonic and non-decreasing function on the interval $[a,b]$. Assume  that $f \geq e$ on this interval. Then there exist a point $x \in [a,\frac{a+b}{2})$ and $N\geq e$ such that 
\begin{equation} \label{eq:mf}
 N \leq f(t) \leq e N
\end{equation} 
for any  $t \in (x-\frac{b-a}{20\log^2 f(x)}, x + \frac{b-a}{20\log^2 f(x)})\subset [a,b].$
\end{lemma}
\begin{proof} Without loss of generality we can assume $a=0, b=1.$
 Define a sequence of numbers $x_i \in  [0,1)$ such that $x_1=0$ and $x_{i+1}=  x_{i} + \frac{1}{10 \log^2 f(x_i)}$ as long as $x_{i+1}<1/2$. The sequence might be finite. Assume that \eqref{eq:mf} fails for $x=\frac{x_{i+1}+ x_i}{2}$ and $N= f(x_i)$, then $f(x_{i+1}) \geq e f(x_i)$. Assuming this for all such $x$,  we obtain $f(x_i) \geq e^i $. Hence  $x_{i+1} -  x_{i} = \frac{1}{10 \log^2 f(x_i)} \leq \frac{1}{10i^2}$. Since $ \sum_{i=1}^{\infty}\frac{1}{10 i^2} < 1/2$, we see that $x_i<1/2$ for all integer $i$ and $f(x_i) \leq f(1/2)$ while $f(x_i) \to \infty$ as $i \to \infty$. 

\end{proof}

We want to apply Lemma \ref{lem:mf} to a modified frequency function:
$$ \tilde \beta(p,r) := \sup_{t \in (0,r]} \beta(p,t).$$
 We note that $\tilde \beta$ is monotonic and 
 $\beta$ and $\tilde \beta$ are comparable due to almost monotonicity of the frequency:
\begin{equation} \label{eq:fm}
 \beta(p,r) \leq \tilde \beta(p,r) \leq  (1+\varepsilon)\beta(p,r), 
\end{equation}
if $B_g(p,r) \subset B_g(O,R_0)$, where $R_0= R_0(\varepsilon,O,M,g)$.
 Hereafter we will work in a small neighborhood of $O$ and always assume that \eqref{eq:fm} holds with $\varepsilon =1$.
 
\begin{lemma} \label{lem:nen}
  Consider a ball $B_g(p,2r) \subset B_g(O,R_0)$ and assume that $\beta(p,r/2) > 10$. Then there exists
 $s \in [r,\frac{3}{2}r)$ and $N\geq 5$ such that
 \begin{equation} \label{eq:nen}
N \leq \beta(p,t) \leq 2e N
\end{equation}
 for any $t \in (s(1 - \frac{1}{1000 \log^2 N}), s(1 + \frac{1}{1000 \log^2 N}))$.
\end{lemma}
 \begin{proof}
 Indeed, we can apply Lemma \ref{lem:mf} for $\tilde \beta(p,t)$ on $[r,2r)$ and find such $s$ and $N $ that 
 $$ 2N<\tilde \beta(p,t) \leq 2eN$$
for $t \in (s - \frac{r}{20 \log^2 (2N)}, s + \frac{r}{20 \log^2 (2N)}).$ By \eqref{eq:fm} we have $N< \beta(p,t) \leq 2eN$ for $t$ on the same interval.

Since $\beta(p,r/2) > 10$ we have $2N\geq 10$. Recall that $s \in [r,2r)$. These two observations imply 
$$ (s(1 - \frac{1}{1000 \log^2 N}), s(1 + \frac{1}{1000 \log^2 N})) \subset (s - \frac{r}{20 \log^2 (2N)}, s + \frac{r}{20 \log^2 (2N)}).$$

 \end{proof}
 \section{Behavior near the maximum}\label{sec:maximum}
 In this section we study the behavior of a harmonic function in the spherical layer of width $\sim \frac{1}{\log^2 N}$ from Lemma  \ref{lem:nen}, where the frequency is comparable to $ N$. We will consider a sphere within this spherical layer and collect several estimates for growth of $u$ near the point, where the maximum is attained on that sphere. 

 The same notation as in Lemma  \ref{lem:nen} is used here: we consider a ball $B(p,2r) \subset B(O,R_0)$ with $\beta(p,r/2) \geq 10$ and a number $s\in[r,2r)$ such that the following holds.
For any $t$ in the interval $$I:=(s(1 - \frac{1}{1000 \log^2 N}), s(1 + \frac{1}{1000 \log^2 N}))$$ the frequency is estimated by $N< \beta(p,t) \leq 2eN$.  We will always assume that  $N$ is larger than $5$.

By $c,c_1,C,C_1,C_2 \dots$ we will denote positive constants that depend on $M,g,n,O,R_0$ only. These constants are allowed to vary from line to line.

 Consider the function $H(p,t)= \int_{\partial B_g(p,t)} u^2$. By \eqref{eq:intbeta} and \eqref{eq:nen} we have 
\begin{equation} \label{eq:t1t2}
 (t_2/t_1)^{ 2N} \leq \frac{H(p,t_2)}{H(p,t_1)} \leq (t_2/t_1)^{4e N}.
\end{equation}
for any $t_1 <t_2$  in $I$.

 Consider a point $x$ on $\partial B_g(p,s)$ such that the maximum of $|u|$ on  $\overline B_g(p,s)$ is attained at $x$ and define $K=|u(x)|$.
 Let us fix numbers 
\begin{equation} \label{eq:s1s2}
A=10^6, \delta \in  [ \frac{1}{A \log^{100} N}, \frac{1}{A \log^2 N}] , s_{-\delta}= s(1 - \delta ), s_{\delta}=s (1 + \delta ).
\end{equation}
 Note that $s_{-\delta}<s<s_{\delta}$ and $\delta <1/10^6$.
\begin{lemma} There exist  $c > 0$ and $C>0$, depending on $M,g,n,O,R_0$ only, such that
\begin{equation} \label{eq:s-}
\sup\limits_{B_g(p, s_{-\delta})} |u| \leq C K 2^{-c\delta N},
\end{equation}
\begin{equation} \label{eq:s+}
\sup\limits_{B_g(p, s_{\delta})} |u| \leq C K 2^{C\delta N}.
\end{equation}
\end{lemma}
\begin{proof}
 We will prove only \eqref{eq:s-}, the same argument works for the second inequality \eqref{eq:s+}.

 By the standard estimate of $L^2$ norm of a function by $L^\infty$ norm and by \eqref{eq:t1t2} we have
$$ K^2 \geq C_1 s^{-n+1}H(p,s) \geq C_1 s^{-n+1} H(p,s_{-\delta/2}) (1+\delta/2)^{2N}.$$
  We need an estimate that compares $L^2$-norm of a harmonic function on the boundary of a ball and $L^2$-norm in the ball:

$$s H(p, s_{-\delta/2})= s \int_{\partial B_g(p,s_{-\delta/2})}|u|^2 \geq  C_1 \int_{B_g(p,s_{-\delta/2})}|u|^2.$$
 Let $\tilde x$ be a point  on $\partial B_g(p, s_{-\delta})$, where the maximum is attained. Define $\tilde K = |u(\tilde x)|$.
 Since the volume $|B_g( \tilde x, \frac{\delta}{2} s)| \geq C_2(\delta s)^n$, we have

$$\int_{B_g(p,s_{-\delta/2})}|u|^2 \geq \int_{B_g( \tilde x, \frac{\delta}{2} s)} |u|^2 \geq C_2 (\delta s)^n\dashint_{B_g( \tilde x, \frac{\delta}{2} s)} |u|^2.$$

One can estimate the value of a harmonic function $u$ in the center of a ball by a constant multiple of the average of $|u|$ over the ball, so 
$$\dashint_{B_g( \tilde x, \frac{\delta}{2} s)} |u|^2 \geq (\dashint_{B_g( \tilde x, \frac{\delta}{2} s)} |u| )^2 \geq C_3 |u|^2 ( \tilde x)= C_3 \tilde K^2.$$
 Combining the estimates above  one has 
\begin{equation} \label{eq:MtM}
 K^2 \geq C_4 \delta^{n}(1+\delta/2)^{2N} \tilde K^2.
\end{equation}
Note that $\log (1+\delta/2) \geq \delta/4$ for $\delta \in(0,1/10^6)$, so  
$$(1+\delta/2)^{2N} \delta^n \geq \exp(N \delta/2 + n \log \delta ) = \exp(N \delta/4) \exp(N \delta/4 + n \log \delta ).$$
Using that $\delta \in  [ \frac{1}{A \log^{100} N}, \frac{1}{A \log^2 N}]$ it is easy to show that  $$C_5+N \delta/4 + n \log \delta>0$$ for sufficiently large $C_5=C_5(n)$. Thus
$K^2 \geq C_6 \exp(N \delta/4) \tilde K^2$.

\end{proof}
Now, we can estimate the doubling index near $x$. Define $\N(x,r)$ by
$$2^{\N(x,r)}=  \frac{\sup\limits_{B_g(x,2r)}|u|}{\sup\limits_{B_g( x, r)}|u|}.$$
 One can estimate the growth of a harmonic function in terms of  the doubling index. For any $\varepsilon$ there exist $R_0>0$ and $C>0$ such that 
 for positive numbers $r_1,r_2$ are such that $2r_1\leq r_2$ and $B_g(x,r_2)\subset B_g(O,R_0) $ the following well-known inequality holds (see \cite{LM2}):
\begin{equation} \label{eq:eNe}
\left(\frac{r_2}{r_1}\right)^{\N(x,r_1)(1-\varepsilon) - C} \leq \frac{\sup_{B_g(x,r_2)}|u|}{\sup_{B_g(x,r_1)}|u|} \leq   \left(\frac{r_2}{r_1}\right)^{\N(x,r_2)(1+\varepsilon) +C}.
\end{equation}
 In particular, the doubling index is almost monotonic in the following sense:
 $$\N(x,r_1)(1-\varepsilon) - C \leq \N(x,r_2)(1+\varepsilon) +C.$$

\begin{lemma} There exists $C=C(M,g,n,O,R_0)>0$ such that 
\begin{equation}\label{eq:max2}
\sup\limits_{B_g(x,\delta s)}|u| \leq K 2^{C\delta N + C}
\end{equation}
and  for any $\tilde x$ with $d(x,\tilde x) \leq \frac{\delta}{4} s$
\begin{equation}\label{eq:max1}
\N(\tilde x, \frac{\delta}{4} s) \leq   C\delta N + C,
\end{equation}
\begin{equation}\label{eq:max3}
\sup\limits_{B_g(\tilde x,\frac{\delta s}{10N} )}|u| \geq K 2^{-C\delta N \log N - C}. 
\end{equation}
\end{lemma}
\begin{proof}
 The first estimate \eqref{eq:max2} immediately follows from \eqref{eq:s+} since $B_g(x,\delta s) \subset B_g(p,s(1+\delta))$.

 To establish \eqref{eq:max1} we note that  $$2^{\N(\tilde x, \frac{\delta}{4} s)} = \frac{\sup\limits_{B_g(\tilde x,\delta s/2)}|u|}{\sup\limits_{B_g(\tilde x,\delta s/4)}|u|} \leq \frac{\sup\limits_{B_g(x,\delta s)}|u|}{K} \leq 2^{C\delta N+ C}.$$
 It remains to obtain \eqref{eq:max3}. We will use \eqref{eq:max1} and almost monotonicity \eqref{eq:eNe} of the doubling index:
$$ \frac{\sup\limits_{B_g(\tilde x,\delta s/4)}|u|}{\sup\limits_{B_g(\tilde x,\frac{\delta s}{10N} )}|u|} \leq (40N)^{C_1\N(\tilde x, \frac{\delta}{4} s) + C_1} \leq  2^{C_2 \delta N \log N + C_2 \log N}  \leq 2^{C_3 \delta N \log N + C_3}.$$
In the last inequality we used that $\delta \in  [ \frac{1}{A \log^{100} N}, \frac{1}{A \log^2 N}]$.
Since $\sup\limits_{B_g(\tilde x,\delta s/4)}|u| \geq |u|(x)=K$, the proof of  \eqref{eq:max3} will be completed if we take $C = C_3$.
\end{proof} 
\section{Number of cubes with big doubling index}
 Given a cube $Q$, we will denote $$\sup\limits_{x\in Q, r \leq \textup{diam}(Q)} \log \frac{\sup\limits_{B_g(x,10 n \cdot r)}|u|}{\sup\limits_{B_g(x, r)}|u|}$$  by $N(Q)$ and call it the doubling index of $Q$. This definition is  different than a doubling index for balls but more convenient in the following sense. If a cube $q$ is contained in a cube $Q$, then $N(q)\leq N(Q)$. Furthermore if a cube $q$ is covered by  cubes $Q_i$ with $diam(Q_i) \geq diam(q)$, then $N(Q_i)\geq N(q)$ for some $Q_i$.  

  The following result was proved in \cite{LM2}, where it was applied to upper estimates of the volume of nodal sets. However this result appears to be useful for lower bounds as well.
 \begin{theorem} \label{th1/2}
  There exist  a constant $c>0$ and an integer $A>1$, depending on the dimension $n$ only, and positive numbers $N_0=N_0(M,g,n,O)$, $R=R(M,g,n,O)$ such that for any cube $Q\subset B(O,R)$ the following holds: if we partition $Q$ into $A^n$ equal subcubes, then the number of subcubes with doubling index greater than $\max(N(Q)/(1+c),N_0)$ is less than $\frac{1}{2} A^{n-1}$.
\end{theorem}
 Further we will partition the cube $Q$ into $A^{nk}$ subcubes ($k$ will tend to infinity) and iterate the Theorem \ref{th1/2} for the subcubes.

\textbf{Notations.} Let $A>1$ be the integer from Theorem \ref{th1/2}.
Given an  Euclidean $n$-dimensional cube $Q$ , we  
partition $Q$ into $A^n$ equal subcubes with $1/A$ smaller size than $Q$,  we denote these cubes by $Q_{i_1}$, $i_1=1,2, \dots, A^n$, then partition each  $Q_{i_1}$ into $A^n$ equal subcubes $Q_{i_1,i_2}$, $i_2=1,2, \dots, A^n$ and so on... 
 The collection of all subcubes $Q_{i_1,i_2,\dots, i_k}$ of all sizes we denote by $\mathbb{A}$.

 By $C^{i}_k$ we denote the binomial coefficients  $\frac{k!}{i! (k-i)!}$.

Let $j_1, j_2, j_3 \dots$ be i.i.d. random variables such that
$$\mathbb{P}(j_k = i ) = 1/A^n \textup{  for  } i=1,2, \dots, A^n. $$

 We make a remark that we use the probabilistic notations  because they are simpler than writing "the number of subcubes with..."

 \begin{lemma} \label{l:iterations}
 Let  $c$,$N_0$ be positive numbers.
 Let $N$ be a function from  the set of subcubes $\mathbb{A}$ to $\mathbb{R}_+$ with the following properties.
\begin{itemize}
\item[(i)] N is monotonic with respect to inclusion:
 if $q_1,q_2 \in \mathrm{Q}$ and $q_1 \subset q_2$, then $N(q_1) \leq N(q_2).$
\item[(ii)]  
 $$  \mathbb{P}\left( N(\tilde Q_{j_1})\geq \max(\frac{N(\tilde Q)}{1+c},N_0) \right) \leq \frac{1}{2A} $$
for any cube $\tilde Q \in \mathbb{A}$.
\end{itemize}
Then for any integers $l,k$  with $0\leq l \leq k$, $k \geq 1$ the following holds
\begin{equation} \label{eq:ind}
\mathbb{P}\left( N(Q_{j_1,j_2, \dots,j_{k-1}, j_k})\leq \max(\frac{N(Q)}{(1+c)^{l}},N_0) \right)\geq 
\sum\limits_{i=l}^{k} C^{i}_k\left(\frac{1}{2A}\right)^{k-i}\left(1-\frac{1}{2A}\right)^i
\end{equation}
and for any $\varepsilon>0$ there exist $\sigma>0$ and an integer $k_0$ such that 
\begin{equation} \label{eq:eps}
 \mathbb{P}\left( N(Q_{j_1,j_2, \dots, j_k})\geq \max(\frac{N(Q)}{(1+c)^{\sigma k/\log k}},N_0) \right) \leq  \left(\frac{1}{2A}\right)^{k(1-\varepsilon)}
\end{equation}
for all positive integers $k>k_0$.
 \end{lemma}

 Before we start the proof of  Lemma \ref{l:iterations} we give some informal explanations.

 \textbf{Heuristics.} Let $p=\frac{1}{2A}$. 
Suppose we have $k$ i.i.d variables $y_i$, each  $y_i$ takes value $0$ with probability
  $p$ and value $1$ with probability $(1-p)$, then 
 $$\mathbb{P}(\sum_{i=1}^k y_i \geq l) =  \sum\limits_{i=l}^{k} C^{i}_k p^{k-i}\left(1-p\right)^i  $$
for $0\leq l \leq k $.

Suppose now that $y_i$ are independent variables, each  $y_i$ takes only  two values $0$ and $1$, 
$\mathbb{P}(y_i=0) \leq p$ and $\mathbb{P}(y_i=1) \geq 1-p$. Now, $y_i$ are not assumed to be identically distributed.
Then 
\begin{equation} \label{eq:binom}
\mathbb{P}(\sum_{i=1}^k y_i \geq l) \geq  \sum\limits_{i=l}^{k} C^{i}_k p^{k-i}\left(1-p\right)^i.
\end{equation}
 The proof of  \eqref{eq:ind} is parallel to \eqref{eq:binom} with the exception  that we have to always  add words ``or smaller than $N_0$''.
 Namely, starting with a cube $ Q_{i_1,\dots, i_k}$ and choosing randomly its subcube 
$ Q_{i_1,\dots, i_k,j_{k+1}}$ the doubling index of the latter is either  $(1+c)$ times smaller than the doubling index of $ Q_{i_1,\dots, i_k}$ or smaller than $N_0$ 
with probability at least $1-p$.

 The inequality \eqref{eq:eps} will be proved with the help of the following fact.

 \textbf{Claim.} Let $p\in(0,1)$ be a fixed number. Then for any $\varepsilon>0$ there is $\sigma>0$ and $k_0>0$ such
 that  
\begin{equation}\label{eq:claim}
\sum\limits_{i=0}^{l-1} C^{i}_k p^{k-i}\left(1-p\right)^i \leq p^{k(1-\varepsilon)}
\end{equation}
 for any $k>k_0$ and $l\in [0, \sigma k/\log k]$.
\begin{proof}[Proof of the claim]
 Note that $C^{i}_k \leq k^l$ for $i \leq l$. Hence
$$\sum\limits_{i=0}^{l-1} C^{i}_k p^{k-i}\left(1-p\right)^i \leq l k^l p^{k-l}.$$
 It sufficient to choose $\sigma>0$ so that 
$$ l k^l p^{k-l} \leq  p^{k(1-\varepsilon)}$$
for large $k$, which is equivalent to
 $$  l k^l (1/p)^l \leq (1/p)^{\varepsilon k }.$$
 Since $l \leq \sigma k/\log k$, we have
$$ l \leq e^{\log k} \leq (1/p)^{\frac{\varepsilon}{3} k }, \quad k^l \leq e^{\sigma k} \leq (1/p)^{\frac{\varepsilon}{3} k } , \quad (1/p)^l \leq (1/p)^{\sigma k/ \log k} \leq (1/p)^{\frac{\varepsilon}{3} k }$$
for $k$ large enough and $\sigma<  \frac \varepsilon 3 \log (1/p)$.
 Multiplying the three inequalities above we finish the proof of the claim.
\end{proof}
\begin{proof}[Proof of Lemma \ref{l:iterations}]
 We are going to prove inequality   \eqref{eq:ind} by induction on $k$.
 For $k=1$ it is true due to assumption (ii). Assume that \eqref{eq:ind} holds for $(k-1)$ in place of $k$ and we want to show that   \eqref{eq:ind} holds for $k$. For $k>l>0$ define the disjoint events $$E_{l,k}:=\left\{ N(Q_{j_1,j_2, \dots,j_{k-1}, j_k})\in \left(\max(\frac{N(Q)}{(1+c)^{l+1}},N_0),\max(\frac{N(Q)}{(1+c)^{l}},N_0) \right]\right\},$$  
   $$E_{0,k}:= \left\{ N(Q_{j_1,j_2, \dots,j_{k-1}, j_k})\in \left( \max(\frac{N(Q)}{(1+c)},N_0),\max(N(Q),N_0) \right]\right\},$$
$$E_{k,k}:= \left\{ N(Q_{j_1,j_2, \dots,j_{k-1}, j_k}) \leq  \max(\frac{N(Q)}{(1+c)^k},N_0)\right\}.$$

 If $l<0$ or $l>k$ we will denote by $E_{l,k}$ the empty event.

Doubling index of any cube is non-strictly greater than the doubling index of any its subcube. Hence $E_{i,k} \subset \cup_{j=0}^i E_{j,k-1} $ and $E_{j,k-1} \subset \cup_{i=j}^k E_{i,k}$, where both unions are disjoint. Hence
$$\mathbb{P}\left( N(Q_{j_1,j_2, \dots,j_{k-1}, j_k})\leq \max(\frac{N(Q)}{(1+c)^{l}},N_0) \right)= \sum_{i=l}^k \mathbb{P}(E_{i,k}) $$
We start to prove by induction on $k$ that 
$$\sum_{i=l}^k \mathbb{P}(E_{i,k}) \geq \sum\limits_{i=l}^{k} C^{i}_{k}\left(\frac{1}{2A}\right)^{k-i}\left(1-\frac{1}{2A}\right)^i. $$
Indeed,
$$ \sum_{i=l}^k \mathbb{P}(E_{i,k})   =\sum_{i=l}^k \sum_{j=0}^i \mathbb{P}(E_{j,k-1}\cap E_{i,k}) = \sum_{j=0}^k \sum_{i=\max(l,j)}^k \mathbb{P}(E_{j,k-1}\cap E_{i,k})  $$

 $$\geq \sum_{j=l}^{k-1} \sum_{i = j}^k \mathbb{P}(E_{j,k-1}\cap E_{i,k}) + \sum_{i=l}^{k} \mathbb{P}(E_{l-1,k-1}\cap E_{i,k})  $$
 $$= \sum_{j=l}^{k-1} \mathbb{P}(E_{j,k-1}) + \sum_{i=l}^{k} \mathbb{P}(E_{l-1,k-1}\cap E_{i,k})= I + II. $$

It follows from (ii) that $$\mathbb{P}(E_{l-1,k-1}\cap E_{l-1,k}) \leq \frac{1}{2A} \mathbb{P} (E_{l-1,k-1}).$$
Since $\sum_{i=l-1}^{k} \mathbb{P}(E_{l-1,k-1}\cap E_{i,k}) = \mathbb{P} (E_{l-1,k-1})$, we obtain 
$$II = \mathbb{P} (E_{l-1,k-1}) -  \mathbb{P}(E_{l-1,k-1}\cap E_{l-1,k})\geq (1-\frac{1}{2A}) \mathbb{P} (E_{l-1,k-1}).$$
 Hence

 $$ I+II \geq \sum_{j=l}^{k-1} \mathbb{P}(E_{j,k-1}) + (1-\frac{1}{2A}) \mathbb{P} (E_{l-1,k-1})$$

 $$ =\frac{1}{2A}\sum_{j=l}^{k-1} \mathbb{P}(E_{j,k-1}) + (1-\frac{1}{2A}) \sum_{j=l-1}^{k-1} \mathbb{P}(E_{j,k-1}).$$

By the induction hypothesis for $k-1$ we can estimate the latter amount from below by

$$  \sum\limits_{i=l}^{k-1} C^{i}_{k-1}\left(\frac{1}{2A}\right)^{k-i}\left(1-\frac{1}{2A}\right)^i +  \sum\limits_{i=l-1}^{k-1} C^{i}_{k-1}\left(\frac{1}{2A}\right)^{k-1-i}\left(1-\frac{1}{2A}\right)^{i+1}$$

$$= \sum\limits_{i=l}^{k} (C^{i}_{k-1}+C^{i-1}_{k-1})\left(\frac{1}{2A}\right)^{k-i}\left(1-\frac{1}{2A}\right)^i=
 \sum\limits_{i=l}^{k} C^{i}_{k}\left(\frac{1}{2A}\right)^{k-i}\left(1-\frac{1}{2A}\right)^i.$$
Inequality \eqref{eq:ind} is proved, which implies

 $$\mathbb{P}\left( N(Q_{j_1,j_2, \dots,j_{k-1}, j_k}) > \max(\frac{N(Q)}{(1+c)^{l}},N_0) \right) \leq \sum\limits_{i=0}^{l-1} C^{i}_k\left(\frac{1}{2A}\right)^{k-i}\left(1-\frac{1}{2A}\right)^i.$$

 It remains to prove \eqref{eq:eps}. 
 By \eqref{eq:claim} applied for $p=\frac{1}{2A}$ we have 
$$\sum\limits_{i=0}^{l-1} C^{i}_k\left(\frac{1}{2A}\right)^{k-i}\left(1-\frac{1}{2A}\right)^i \leq \left (\frac{1}{2A} \right)^{k(1-\varepsilon)}$$

for $0\leq l\leq \sigma k / \log k $, $k \geq k_0$.

\end{proof}
 
 Now, we ready to formulate the corollary of Theorem \ref{th1/2} and Lemma \ref{l:iterations},
which will be used in the next section.

  \begin{theorem} \label{th:blogb}
 There exist constants $c_1, c_2, C>0$ and a positive integer $B_0$, depending on the dimension $n$ only, and positive numbers $ N_0=N_0(M,g,n,O)$, $R=R(M,g,n,O)$  such that for any cube $Q\subset B(O,R)$ the following holds: if we partition $Q$ into $B^n$ equal subcubes, where $B>B_0$, then the number of subcubes with doubling index greater than $\max(N(Q)2^{-c_1 \log B/ \log \log B},N_0)$ is less than $ C B^{n-1 - c_2}$.
  \end{theorem}
 \begin{proof}
  Let us fix $c,A,N_0,R$ from Theorem  \ref{th1/2}.
  
 Fix a cube $Q\subset B(O,R)$ and partition it into $A^n$ equal subcubes $Q_{i_1}$, then each $Q_{i_1}$
partition into $A^n$ subcubes $Q_{i_1,i_2}$ and so on. The collection of all subcubes $Q_{i_1,i_2,\dots, i_k}$
of all sizes we denote by $\mathbb{A}$.

First, we will consider the  case $B=A^k$, where $k$ is sufficiently large. 
In this case  Theorem  \ref{th:blogb} follows from Lemma  \ref{l:iterations}.
 Let's first check  the assumptions (i) and (ii) of Lemma \ref{l:iterations}.
 The monotonicity property (i) for the  doubling index of cubes is clear from the definition.
 The second assumption (ii) follows from Theorem \ref{th1/2}.
 Now fix $\varepsilon>0$ so small that
 $$\left(\frac{1}{2A}\right)^{1-\varepsilon}=\left (\frac{1}{A} \right )^{1+c_2}$$
for some $c_2>0$.
 
The conclusion \eqref{eq:eps} of Lemma \ref{l:iterations} for this $\varepsilon$
claims
that the number of subcubes $Q_{i_1,i_2,\dots, i_k}$ with $N(Q_{i_1,i_2, \dots, i_k})\geq \max(\frac{N(Q)}{(1+c)^{\sigma k/\log k}},N_0) $ is smaller than  $$ A^{nk}  \left(\frac{1}{2A}\right)^{k(1-\varepsilon)}= A^{k(n-1-c_2)}=B^{n-1-c_2}.$$
 Note that $\log B = k \log A$, we therefore can choose $c_1>0$ so small that $$\log(1+c)\cdot \sigma k/\log k \geq c_1 \log 2 \cdot \log B/ \log\log B$$ for all sufficiently large $B=A^k$. This is done to provide
$$ (1+c)^{\sigma k/\log k} \geq 2^{c_1 \log B/ \log\log B}.$$

 We have proved Theorem \ref{th:blogb} in the case $B=A^k$. 

Now, let $B \in [A^{k},A^{k+1}]$ and define $\tilde B = A^{k}$. There are two partitions of $Q$ into equal subcubes, say $Q=\cup Q_i$, $i=1\dots B^n$, and $Q=\cup \tilde Q_i$,
, $i=1\dots \tilde B^n$. We know that the number   of cubes $\tilde Q_i$ with doubling index greater than  
$\max(N(Q)2^{-c_1 \log \tilde B/ \log \log \tilde B},N_0)$ is less than  $\tilde B^{n-1 - c_2}$.
Each cube $Q_i$ is covered by a finite number, which depends on dimension $n$ and on $A=A(n)$ only,   of cubes $\tilde Q_j$, which have a smaller diameter.
If  $N(Q_i)$ is greater than $\max(N(Q)2^{-c_1 \log \tilde B/ \log \log \tilde B},N_0)$, then one of $\tilde Q_j$ that cover $Q_i$ also has $N(\tilde Q_j)$ greater than $\max(N(Q)2^{-c_1 \log \tilde B/ \log \log \tilde B},N_0)$.
Thus the number of cubes $ Q_i$ with doubling index greater than  $\max(N(Q)2^{-c_1 \log \tilde B/ \log \log \tilde B},N_0)$ is less than $C\tilde B^{n-1 - c_2}$.  We can decrease $c_1$ and increase $C$ to replace $\tilde B$ by $B$ in the previous sentence.

  \end{proof}

\begin{remark}
Here we collected several informal remarks to orient the reader.
 The goal of this paper is to estimate 
 the Hausdorff measure  of dimension $n-1$ of zero sets of harmonic functions from below. 
 If a harmonic function is zero at the center of a cube and the doubling index of this cube is bounded by a fixed constant, then it is not difficult and well known 
that there is a lower bound for the volume of the zero set  in this cube. Unfortunately, the bound depends on the doubling index and  it is not clear  why the lower estimate 
 does not become worse as the doubling index becomes large.

 In the next section there will be an argument that works for the case of the large doubling index of the original cube. Speaking non-formally the argument   will show that for a proper choice of $B$ the number of  subcubes, which contain zeroes, is larger than $B^{n-1}$ , and the argument  severely exploits  that the number of  bad subcubes with large doubling index is smaller than $B^{n-1}$. 
 We don't specify here what the words ``smaller'' and ``larger" mean.

 If we partition the cube with zero at the center into $B^n$ equal subcubes, there can be some subcubes
 with small doubling index, which intersect the zero set, but there also can be bad
 subcubes with large doubling index, where we have no good a priori estimate. 
 The estimate for the number of bad subcubes appears to be useful.

 In Theorem \ref{th1/2} the number of subcubes $A^n$ is fixed and it shows that all except at most $\frac{1}{2}A^{n-1}$ of the subcubes have constant times smaller doubling index than a big cube. For the estimates of the volume of the nodal set it is crucial that the number of exceptions is smaller  than  $A^{n-1}$.  In Theorem \ref{th:blogb}  the number of subcubes $B^n$  tends to infinity, but the bigger $B$ the smaller the doubling index for the most of the subcubes becomes and we still want the number of bad subcubes with big doubling index to be smaller than $B^{n-1}$.  Theorem \ref{th:blogb} is the iterated version
of Theorem \ref{th1/2},  the iteration procedure is similar to the independent flips of the coin. The  quantity  $ k/ \log k \sim \log B / \log \log B$ in  Theorem \ref{th:blogb}   comes from the simple estimate of the tails of the binomial distribution \eqref{eq:claim}.

 We also note that for the purposes of the paper  a weaker estimate than the conclusion of Theorem \ref{th:blogb} would be sufficient.
 Namely, it is sufficient to know that the number of subcubes with doubling index greater than $\max(N(Q)/ (\log B)^\kappa ,N_0)$ is less than $  B^{n-1}/ (\log B)^\kappa$, where $\kappa>0$ is a sufficiently large constant depending only on the dimension.  

\end{remark}

\section{A tunnel with controlled growth}
 This section contains a geometrical construction  that allows to find many disjoint balls 
with sign changes of the harmonic function (Proposition \ref{cor:signch}). It appears to be useful for lower estimates for the nodal sets.   The construction is using the estimates for the number of cubes with big doubling index and requires  to look at several statements of the previous sections. The whole 
section  consists of the proof of one proposition.
\begin{prop} \label{cor:signch}
Fix a point $O $  on the Riemannian manifold $M$ equipped with Riemannian metric $g$.  There is  a sufficiently small radius $R_0>0$ such that for any ball $B_g(p,2r) \subset B_g(O,R_0)$ and for any harmonic function $u$ on $B_g(p,2r)$ the following holds. If $\beta(p,r)$ is sufficiently large, then there is
a number $N$ with 
$$ \beta(p,r)/10 \leq N \leq 2\beta(p,\frac{3}{2}r) $$ and  at least $[\sqrt N]^{n-1} 2^{c_3 \log N / \log\log N}$ disjoint balls ${B_g(x_i, \frac{r}{\sqrt N})} \subset B(p,2r)$  such that $u(x_i)=0$.
\end{prop}
  \begin{proof}

 According to Section \ref{sec:maximum} we can find a spherical layer where the frequency does not grow too fast:
there exist numbers $s \in[r,\frac{3}{2}r]$  and $N \geq 5$ such that
$$ N \leq \beta(p,t) \leq 2eN$$
for any $t \in(s(1 - \frac{1}{1000 \log^2 N}), s(1 + \frac{1}{1000 \log^2 N}))$.

By the monotonicity property of the frequency $$\beta(p,r) \leq (1+\varepsilon) \beta(p,t) \leq 10 N$$
and $$N \leq \beta(p,s) \leq (1+\varepsilon) \beta(p,\frac{3}{2}r) \leq 2 \beta(p,\frac{3}{2}r). $$

 Till the end of this section we will assume that $N$ is sufficiently large. 

 Fix a point $x \in \partial B_g(p,s)$  such that $ \sup\limits_{\partial B_g(p,s)} |u| = |u(x)|$.

Put 
\begin{equation} \label{eq:delta}
\delta =  \frac{1}{10^8 n^2\log^{2} N}.
\end{equation}

 Consider a point $\tilde x \in \partial B_g(p,s(1-\delta))$ such that $d_g(x,\tilde x) = \delta s$. In other words, $\tilde x$ is the nearest point to $x$ on $\partial B_g(p,s(1-\delta))$.  Note that
\begin{equation} \label{eq:size}
C_1(n) \frac{r}{\log^2 N} \leq d(x,\tilde x) \leq C_2(n) \frac{r}{\log^2 N}. 
\end{equation}

Let us consider a box $T$ (a hyperrectangle in the Euclidean space)  such that $x$ and $\tilde x$ are the centers  of the opposite faces of $T$,   one side of $T$ is equal to $d(x,\tilde x)$ and $n-1$ other sides are equal to $\frac{d(x,\tilde x)}{[\log N]^{4}}$, where $[ \cdot]$ denotes the integer part of a number. 

Let us divide $T$ into equal boxes $T_i$, $i =1,2,\dots ,[\sqrt N]^{n-1}$, so that each $T_i$ has one side of length $d(x,\tilde x)$ and (n-1) sides of length $\frac{d(x,\tilde x)}{[\sqrt N] [\log N]^{4}}$.
 We partition each $T_i$ into equal cubes $q_{i,t}$, $t=1,2, \dots, [\sqrt N] [\log N]^{4}$, with side $\frac{d(x,\tilde x)}{[\sqrt N] [\log N]^{4}}$, and the cubes $q_{i,t}$ are arranged in $t$ so that  $d(q_{i,t},x) \geq d(q_{i,t+1},x)$. We will call the boxes $T_i$ "tunnels".
 
 Note that $$d_g(p, q_{i,1}) \leq d_g(p, \tilde x) + d_g( \tilde x, q_{i,1}) \leq s(1-\delta) +  C \frac{ \delta s \sqrt n }{ [\log N]^4} \leq s(1-\delta/2 ).$$
 
 Hence $q_{i,1} \subset B_g(p, s(1-\delta/4)) $.
Recall that $|u(x)|=K$. Then by \eqref{eq:s-} 
\begin{equation} \label{tunnel1}
 \sup\limits_{q_{i,1}} |u|\leq K 2^{-c_1 \frac{N}{\log^2 N} +C_1}.
\end{equation}

Applying \eqref{eq:max1} with  $\delta$, which is $100n^2$ times larger   than $\delta$ defined by \eqref{eq:delta}, we obtain that
 for any point $y \in T$
\begin{equation} \label{tunnel2}
 \N(y,  10n  \delta s) \leq C \delta N  + C \leq N/100.
\end{equation}

The center of $q_{i,t}$  will be denoted by $x_{i,t}$.

 Now, let $t=[\sqrt N] [\log N]^{4}$. We can inscribe a geodesic ball $B_{i,t}$ in $\frac{1}{2}q_{i,t}$ with center at $x_{i,t}$ and  radius $ \frac{s}{ N}$. Taking into account $$d_g(x_{i,t}, x) \leq {\frac{C_2s}{[\log N]^6}},$$ we deduce from \eqref{eq:max3}, applied with $\tilde x = x_{i,t}$, that  $$\sup\limits_{B_{i,t}}|u| \geq K 2^{-C_3 \frac{N}{\log^{5} N} -C_3},$$
and therefore
\begin{equation} \label{tunnel3}
\sup\limits_{\frac{1}{2}q_{i,[\sqrt N] [\log N]^{4}}}|u| \geq K 2^{-C _3\frac{N}{\log^{5} N} -C_3}.
\end{equation}
  The inequalities \eqref{tunnel1}, \eqref{tunnel3} imply the following estimate:
 there exist positive $c$, $C$ such that
\begin{equation} \label{tunnel4}
\sup\limits_{\frac{1}{2}q_{i,[\sqrt N] [\log N]^{4}}}|u| \geq \sup\limits_{\frac{1}{2}q_{i,1}} |u| 2^{c N/ \log^2N - C}.
\end{equation}

 The next step in the proof of Proposition \ref{cor:signch} is the following claim.
\begin{claim} \label{le:goodt} There exist $c>0$, $N_0>0$ such that 
at least  half of  tunnels $T_i$ have the following property
\begin{equation} \label{eq:goodt}
N(q_{i,t}) \leq \max( \frac{N}{2^{c \log  N/ \log \log  N}}, N_0)
\end{equation}
for all $t=1,2, \dots [\sqrt N] [\log N]^{4}$.
\end{claim}
\begin{proof}[Proof of the claim]
We will assume that $N$ is sufficiently big.
Let us call a cube $q_{i,t}$ bad if $N(q_{i,t}) >  N2^{-c_1 \log  N/ \log \log  N}$, where a constant $c_1$ is from Theorem \ref{th:blogb}. It is sufficient to show that the number of bad cubes is less than the half of the number of tunnels $T_i$, i.e. $\frac{1}{2}[\sqrt N]^{n-1}$.

Let us partition $T$ into equal Euclidean cubes $Q_{t}, t=1,2, \dots, [\log N]^{4}$ with side $\frac{d(x,\tilde x)}{[\log N]^{4}}$. 
For any point $y \in T$ the distance $$d_g(x,y) \leq 2d(x,y) \leq 4 d(x,\tilde x) \leq \frac{s}{10^7  \log^2N}=: \rho.$$
 By \eqref{eq:max1} we have $$\frac{\sup\limits_{B_g(y,\rho)}|u|}{\sup\limits_{B_g(y,\rho/2)}|u|} \leq 2^{C N / \log^2 N + C}.$$ The last observation implies that  $$N(Q_{t}) \leq N$$
 for  $t=1,2, \dots, [\log N]^{4}$.

 It follows from Theorem \ref{th:blogb} with $B=[\sqrt N]$ that the number of bad cubes in $Q_{t}$ is less than $C[\sqrt N]^{n-1-c_2}$. Thus the number of all bad cubes is less than 
$$C[\sqrt N]^{n-1-c_2} [\log N]^4 \leq \frac{1}{2}[\sqrt N]^{n-1}.$$

\end{proof}

 We will call a tunnel $T_i$ good if \eqref{eq:goodt} holds.

The next step in the proof of Proposition \ref{cor:signch} is the following claim.
\begin{claim}
There exists $c_2>0$ such that if  $N$ is sufficiently large and $T_i$ is a good tunnel, then 
there are at least $2^{c_2 \log N /\log\log N }$ closed cubes $\overline{q_{i, t}}$ that contain zero of $u$.
\end{claim}
\begin{proof}[Proof of the claim]
 By \eqref{eq:goodt} we know that  

  \begin{equation} \label{tunnel5}
\log \frac{\sup\limits_{\frac{1}{2}q_{i,t+1}}|u|}{\sup\limits_{\frac{1}{2}q_{i,t}}|u|} \leq 
\log \frac{\sup\limits_{4q_{i,t}}|u|}{\sup\limits_{\frac{1}{2}q_{i,t}}|u|}
 \leq \frac{N}{2^{c_1 \log N / \log\log N} }
\end{equation}
for any $t=1,2, \dots, [\sqrt N] [\log N]^{4}-1$.

Let us split the set $\{1,2, \dots, [\sqrt N] [\log N]^{4}-1\}$ into two subsets $S_1, S_2$. The set $S_1$ is the set of all $t$ such  that $u$ does not change the sign in $\overline{q_{i, t}} \cup \overline{q_{i, t+1}} $
and $S_2 = \{1,2, \dots, [\sqrt N] [\log N]^{4}-1\} \setminus S_1 $.
 By the Harnack inequality for $t \in S_1$ we have 
\begin{equation}
\log \frac{\sup\limits_{\frac{1}{2}q_{i,t+1}}|u|}{\sup\limits_{\frac{1}{2}q_{i,t}}|u|} \leq C_1. 
\end{equation}
 and for any $t \in S_2$ the inequality \eqref{tunnel5} holds.
 We therefore have  $$ \log \frac{\sup\limits_{\frac{1}{2}q_{i, [\sqrt N] [\log N]^{4}}}|u|}{\sup\limits_{\frac{1}{2}q_{i,1}}|u|} = \sum\limits_{S_1} \log \frac{\sup\limits_{\frac{1}{2}q_{i,t+1}}|u|}{\sup\limits_{\frac{1}{2}q_{i,t}}|u|} + \sum\limits_{S_2} \log \frac{\sup\limits_{\frac{1}{2}q_{i,t+1}}|u|}{\sup\limits_{\frac{1}{2}q_{i,t}}|u|} \leq  $$
 $$ |S_1| C_1 +|S_2| \frac{N}{2^{c_1 \log N / \log\log N}}.$$
 By \eqref{tunnel4}  
\begin{equation} \label{tunnel7}
  c N/\log^2 N - C \leq  \log \frac{\sup\limits_{\frac{1}{2}q_{i,[\sqrt N] [\log N]^{4}}}|u|}{\sup\limits_{\frac{1}{2}q_{i,1}}|u|}.
\end{equation}
 Hence 
 $$ c N/\log^2 N - C \leq  |S_1| C_1 +|S_2| \frac{N}{2^{c_1 \log N / \log\log N}}.$$
 Note that  $$ |S_1| C_1 \leq C_1 [\sqrt N] \log^4 N \leq \frac{c}{2} N/\log^2 N - C$$ 
 for $N$ large enough.
 Thus 
$$| S_2| \geq \frac{2^{c_1 \log N / \log\log N}}{2 \log^2 N} \geq 2^{\frac{c_1}{2} \log N / \log\log N}.$$
\end{proof}
 We continue the proof of Proposition \ref{cor:signch}.
  At least half of the tunnels $T_i$ are good by Claim \ref{le:goodt}. Hence the number of cubes $\overline{q_{i,t}}$ where $u$ changes a sign is at least 
 $\frac{1}{2}[\sqrt N]^{n-1} 2^{c_2 \log N / \log\log N}$. For any such cube let us fix a point  $x_{i,t} \in \overline{q_{i,t}}$
 such that $u(x_{i,t})=0$. We had find many disjoint cubes with sign changes. To replace the cubes by balls is not difficult.

 The fact that the side of $\overline{q_{i,t}}$ is comparable to $\frac{r}{\sqrt N \log^6 N}$ shows that
each ball $B_g(x_{i_0,t_0}, \frac{r}{\sqrt N })$ intersects not greater than $C_3 [\log N]^{6n}$ other balls $B_g(x_{i,t}, \frac{r}{\sqrt N })$. We can choose the maximal set of disjoint balls $B_g(x_{i,t}, \frac{r}{\sqrt N })$.
 Since the number of $x_{i,t}$ is at least $\frac{1}{2}[\sqrt N]^{n-1} 2^{c_2 \log N / \log\log N} $ and the number of intersections for each ball is bounded by  $C_3 [\log N]^{6n}$,  the maximal set of disjoint balls $B_g(x_{i,t}, \frac{r}{\sqrt N })$ will consist of at least $\frac{c_3[\sqrt N]^{n-1} 2^{c_2 \log N / \log\log N}}{ [\log N]^{6n}}$ balls. We can choose $c_4 \in (0,c_2)$
such that  $$\frac{c_3[\sqrt N]^{n-1} 2^{c_2 \log N / \log\log N}}{ [\log N]^{6n}} \geq [\sqrt N]^{n-1} 2^{c_4 \log N / \log\log N}$$
for large enough $N$.

\end{proof}

\begin{remark} The following remark will not be used later, but shows the flexibility of the construction.
 In the statement and in the proof of Proposition \ref{cor:signch} one can replace $\sqrt N$ by $N^{\alpha}$ with any $\alpha \in (0,1)$ and the statement will remain true.

If we fix a point $O $  on the Riemannian manifold $M$ equipped with Riemannian metric $g$.  There is  a sufficiently small radius $R_0>0$ such that for any ball $B_g(p,2r) \subset B_g(O,R_0)$ and for any harmonic function $u$ on $B_g(p,2r)$ the following holds. If $\beta(p,r)$ is sufficiently large, then there is
a number $N$ with 
$$ \beta(p,r)/10 \leq N \leq 2\beta(p,\frac{3}{2}r) $$  and  at least $N^{\alpha(n-1)} 2^{c \log N / \log\log N}$ disjoint balls ${B_g(x_i, \frac{r}{ N^\alpha})} \subset B(p,2r)$  such that $u(x_i)=0$.
\end{remark}

\section{Estimate of the volume of the nodal set}
 In this section we prove Theorem \ref{Nad}, we formulate it in the scaled form.

Define the function
$$F(N):= \inf  \frac{H^{n-1}(\{u=0\}\cap B_g(x,\rho))}{ \rho^{n-1}}$$ 
 where the infimum is taken over all balls $B_g(x,\rho)$ within $B_g(O,R_0)$ and all harmonic functions $u$ on $M$ with respect to  metric $g$ such that $u(x)=0$ and $N(B_g(O,R_0))\leq N$. Here we denote by the $N(B_g(O,R_0))$ the supremum of
$\beta(x,r)$ over all $B_g(x,r) \subset B_g(O,R_0) $. Recall that the radius $R_0= R_0(M, g, n, O)$ is a sufficiently small positive number. 
 \begin{theorem} \label{Nad1} There exists $c>0$ such that 
$F(N) \geq c$ for all positive $N$.
 \end{theorem}
 
\begin{proof}
 
Let $u$ be a harmonic function that vanishes at $x$, $B_g(x,\rho) \subset B(O,R_0)$ and $\beta_u(x,\rho) \leq N$ for $B_g(x,\rho) \subset B(O,R_0)$.
 By the allmost monotonicity property of the doubling index we know that 
$$N \geq \frac{1}{2} \lim\limits_{\rho \to 0} \beta(x,\rho) \geq 1/2.$$ 
 Hence $N$ is separated from zero.
 Furthermore let us assume that $F(N)$ is almost attained on $u$:
\begin{equation} \label{eq:contradiction}
\frac{H^{n-1}(\{u=0\}\cap B_g(x,\rho))}{ \rho^{n-1}} \leq 2F(N).
\end{equation}
We start with a  naive and well-known estimate that gives some lower bound for $F(N)$.
There exists $c_1>0$ such that  
\begin{equation} \label{eq:trivial}
\frac{H^{n-1}(\{u=0\}\cap B_g(x,\rho))}{ \rho^{n-1}} \geq \frac{c_1}{(\beta(x,\rho/2))^{n-1}} \geq  \frac{c_2}{N^{n-1}}.
\end{equation}
This estimate follows from the fact that if a harmonic function $u$ vanishes at $x$ and has the frequency (or the doubling index) of  $B_g(x,\rho/2)$  smaller than $N$, then
 one can inscribe in $B_g(x,\rho/2)$ a ball of radius $\sim \frac{\rho}{N}$ where $u$ is  positive
 and a ball  of radius $\sim \frac{\rho}{N}$ where $u$ is  negative.  For instance, see \cite{LM1} for the details.

 We can use the estimate \eqref{eq:trivial} to bound $F(N)$ from below for small  $\beta(x,\rho/2)$.
 Now, we will assume that $N$ is sufficiently big and will show that $\beta(x,\rho/2)$ is bounded.

  We argue by assuming the contrary.
Let $\beta(x,\rho/2)$ be sufficiently big, then
we can apply Proposition \ref{cor:signch} for the ball $B_g(x, 2r)= B_g(x, \rho)$  and find a number $$ \tilde N \geq \beta(x,\rho/2)/10$$ and $[\sqrt{\tilde N}]^{n-1} 2^{c_3 \log \tilde N / \log\log \tilde N}$ disjoint balls $B_g(x_i, \frac{r}{\sqrt{\tilde N} })$ within $ B(x,2r)$  such that $u(x_i)=0$.
For each $i$ we know that 
$$H^{n-1}(\{u=0\}\cap B_g(x_i, \frac{r}{\sqrt{\tilde N} })) \geq F(N) \left(\frac{r}{\sqrt{\tilde N} } \right)^{n-1}.$$
 Since the number of such balls is at least $[\sqrt{\tilde N}]^{n-1} 2^{c_3 \log \tilde N / \log\log \tilde N}$ and these balls are disjoint and contained in $B_g(x,\rho)$, we get 

$$H^{n-1}(\{u=0 \} \cap B_g(x,\rho)) \geq F(N) \left[\sqrt{\tilde N}\quad \!\!\! \right]^{n-1} 2^{c_3 \log \tilde N / \log\log \tilde N} \left(\frac{\rho/2}{\sqrt{\tilde N} } \right)^{n-1}  $$
 We can decrease $c_3$ to a smaller positive constant $c_4$ such that
\begin{equation} \label{eq:loglog}
 H^{n-1}(\{u=0 \} \cap B_g(x,\rho)) \geq 2^{c_4 \log \tilde N / \log\log \tilde N} F(N) \rho^{n-1}
\end{equation} 
if $N$ is sufficiently large. 
The last observation contradicts to \eqref{eq:contradiction}.

We have proved that $\tilde N$ is bounded from above by some positive constant $N_0$ and we can use \eqref{eq:trivial} with $\rho=r$ to obtain the uniform bound $F(N) \geq \frac{c_5}{N^{n-1}_0}$.
 
\end{proof}
\begin{remark}
Now, we know that $F(N)$ is uniformly bounded from below. Since $ \beta(x,\rho/2)/10 \leq  \tilde N$, the inequality \eqref{eq:loglog} implies the following estimate of the volume of the nodal set:

 $$ \frac{H^{n-1}(\{u=0 \} \cap B_g(x,\rho))}{\rho^{n-1}} \geq  2^{c_6 \log \beta(x,\rho/2)/ \log\log \beta(x,\rho/2)} $$
for $\beta(x,\rho/2) > \beta_0$.
\end{remark}

\section{The lower bound in Yau's conjecture} \label{sec:Yau}
In this section we prove Theorem \ref{Yau}.

Let $B$ be a geodesic ball of fixed radius on a Riemannian manifold $M$.
Consider a function $u$ on $B$ that satisfies $\Delta u +\lambda u = 0$ in $B$ and
  the harmonic extension of $u$
$$h(x,t)= u(x)\exp(\sqrt \lambda t).$$

   The following lemma is well-known, but for the convenience of the reader we give the proof below.
\begin{lemma} \label{le:dense}
 There exists $C_1>0$ and $\lambda_0>0$, depending on $M$ and $B$ only such that
 if $\lambda> \lambda_0$ then $Z_u$ is $\frac{C_1}{\sqrt \lambda}$ dense in $B$.
\end{lemma}
\begin{proof}
 Let $y$ be  a point in $ B \times [-1,1]$. Denote the geodesic ball with center at $y$ and radius $r$ on $M\times \mathbb{R}$ by $B_{M\times \mathbb{R}}(y,r)$. 
 The Harnack inequality for harmonic functions says that there exist $C_1(M,B) > 1$ and $r_0(M,B)>0$ such that 
 if  $0<r<r_0$ and $h$ is positive on $B_{M\times \mathbb{R}}(y,r)$, then for any $\tilde y \in B_{M\times \mathbb{R}}(y,r/2) $   the following inequality holds:
 $$   h(\tilde y) < C_1 h(y).$$
 Let us formulate the Harnack inequality in the following form:
 if $ |h(\tilde y)| \geq C_1 |h(y)|$, then $h$ changes sign in $B_{M\times \mathbb{R}}(y,r)$.

 Let $C_2 = \log C_1 $. Consider a point $y=(x,0),$ $x\in B$ and the point  $\tilde y=(x, C_2 / \sqrt \lambda)$.
 Since $h(\tilde y)=C_1 h( y)$, by the Harnack inequality we know that 
 if $\lambda$ is sufficiently big and $B_{M\times \mathbb{R}}(y, 3C_2/\sqrt \lambda) \subset B\times [-1,1]$, then there is a point $\tilde{\tilde y} \in B\times [-1,1]$ such that $h(\tilde {\tilde y})=0$ and $d_g(y,\tilde{\tilde y}) \leq 3C_2/\sqrt \lambda$. In other words, $Z_h$ is $\frac{const}{\sqrt \lambda }$ dense in $B\times [-1,1]$.
Since $Z_h= Z_u \times \mathbb{R}$,  the zero set $Z_u$ is also $\frac{C}{\sqrt \lambda}$ dense in $B$.
 
\end{proof}
 Now, it is a straightforward matter to prove Theorem \ref{Yau}.
\begin{proof}
  By Lemma \ref{le:dense} if $\lambda> \lambda_0$, then $Z_u$ is $\frac{C}{\sqrt \lambda}$ dense in $B$ and we can find $c (\sqrt \lambda)^{n}$ disjoint balls $B_M(x_i, C/\sqrt \lambda)$ such that $u(x_i)=0$. It is sufficient to show that 
\begin{equation}\label{eq:n-1}
H^{n-1}(Z_u \cap B_M(x_i, C/\sqrt \lambda)) \geq c_1 \lambda^{-\frac{n-1}{2}}
\end{equation}
for some $c_1=c_1(M,B)>0$. Indeed, since the balls are disjoint, it would immediately give $H^{n-1}(Z_u\cap B) \geq c_2 \sqrt \lambda$. 

We can apply Theorem \ref{Nad1} for the function $h$ to see that $$H^n(Z_h \cap B_{M\times \mathbb{R}}((x_i,0), \frac{C}{\sqrt \lambda})) \geq c_3 \lambda^{-n/2}.$$
In view of $Z_h=Z_u \times \mathbb{R}$ that gives \eqref{eq:n-1}.
 
\end{proof}

\end{document}